\documentclass[12pt]{amsart}

\usepackage[pdftex,hyperindex]{hyperref} 
\usepackage{comment} 
\usepackage{tikz-cd}
\usepackage{tikz}
\usetikzlibrary{positioning, arrows}
\usepackage{amsmath,amsfonts,amstext,amssymb, pifont, stmaryrd,mathabx, relsize}
\usepackage[all,2cell]{xy}
\usepackage{lipsum}
\usepackage{tcolorbox} 
\usepackage{graphicx} 
\usepackage{xcolor}

\newtheorem{theorem}{Theorem}[section]
\newtheorem{lemma}[theorem]{Lemma}

\newtheorem{proposition}[theorem]{Proposition}
\theoremstyle{definition}
\newtheorem{definition}[theorem]{Definition}

\theoremstyle{remark}
\newtheorem{remark}[theorem]{Remark}

\numberwithin{equation}{section}

\newenvironment{smatrix} 
  {\left(\begin{smallmatrix}}
  {\end{smallmatrix}\right)
}

\newcommand{\field}{\mathsf{k}}
\newcommand{\dg}{\text{DG}}
\newcommand{\dgcat}{\text{DG-Mod}}
\newcommand{\dga}{\text{DGA}}
\newcommand{\dgacat}{\text{DGA-Mod}}

\newcommand{\sset}{\textbf{sSet}}
\newcommand{\set}{\textbf{Set}}
\newcommand{\smodcat}{\mathbb{S}\text{-Mod}}

\begin{document}

\title{E-infinity structures over $\mathcal{L}$-algebras}



\author[Jes\'us S\'anchez Guevara]{Jes\'us S\'anchez Guevara}


\address{Escuela de Matem\'aticas, Universidad de Costa Rica}
\email{jesus.sanchez\_g@ucr.ac.cr}
\thanks{Centro de Investigaci\'on en Matem\'atica Pura y Aplicada CIMPA-UCR}







\begin{abstract}
	In this paper we introduce the concept of $\mathcal{L}$-algebras,
	which can be seen as a generalization of 
	the structure determined by 
	the Eilenberg-Mac lane transformation 
	and Alexander-Whitney diagonal 
	in chain complexes.
	In this sense, our main result states that
	$\mathcal{L}$-algebras are endowed with an
	$E_\infty$-coalgebra struture,
	like the one determined by the Barrat-Eccles
	operad in chain complexes.
	This results implies that 
	the canonical $\mathcal{L}$-algebra of spaces 
	contains as much homotopy information as 
	its usually associated $E_\infty$-coalgebras,
	suggesting $\mathcal{L}$-algebras 
	as a tool for the study of homotopy types.

\end{abstract}

\maketitle

\tableofcontents

\section{Introduction}
\label{sec-introduction}

The central notion of this paper is the algebraic structure called $\mathcal{L}$-algebra. 
Introduced by Alain Prouté in several talks since the eighties and never published
(Max Planck Institut-Bonn 1986, 
Louvain-la Neuve 1987, 
Freie Universität-Berlin 1988,
Seminar Keller-Maltsiniotis-Paris 2010),
$\mathcal{L}$-algebras have been thought to be highly related 
to the homotopy type of spaces by
using an internal structure that models 
the diagonals which determines invariants
like Steenrod operations. 
$\mathcal{L}$-algebras are also inspired in Segal's $\Gamma$ structures (see \cite{SEGAL1974293}),
and can be thinking as a kind of co-version of $\Gamma$ spaces.

Eilenberg-Mac~Lane
transformation plays a central role in $\mathcal{L}$-algebras, where it is
the prototype (motivation) for the product of $\mathcal{L}$-algebras.
It is interesting to notice the existence of a preprint 
of Tom Leinster (see \cite{2000math2180L}), which present a similar object.

Using a homotopy inverse of the structural quasi-isomorphism $\mu$ of an $\mathcal{L}$-algebra $A$ we can
define a coproduct on its main element $A[1]$. Indeed, we only have to take the composition
of an homotopy inverse of $\mu:A[1]\otimes A[1]\to A[2]$ with the morphism $s_1:A[1]\to A[2]$.
Observe that this coproduct in general is not associative. But, the structure of $\mathcal{L}$-algebra
makes this coproduct associative and commutative up to homotopy. Moreover, the homotopies
also satisfy to be associative and commutative up to homotopy, and this property is
maintained on the next levels of homotopies, generating a system of higher homotopies. 
The classical case where this happens is in the context of
chain complexes associated to a simplicial set. 

Higher homotopies can be organized into an $E_\infty$-coalgebra. 
Such a structure was exhibited in particular by Smith 
(see \cite{smith1994iterating}) and alternatively 
by S\'anchez-Guevara using a simplification of Smith's operad (see \cite{sanchezguevara2021einfinity}).
We will generalize these descriptions in the context of $\mathcal{L}$-algebras
with values in the category $\dga$-Mod,
in other words, we will prove that the main element of an $\mathcal{L}$-algebra $A$
is equipped with an $E_\infty$-coalgebra structure describing the system
of higher homotopies associated to the coproducts induced by the structural quasi-isomorphism of $A$.

The main difference with the case of chain complexes associated to a simplicial set,
where the process begins with the Alexander-Whitney diagonal, which is an associative
coproduct, is that in general we don't have the associativity. Then, in order
to model the higher homotopies we have to consider an $E_\infty$-operad that will have
several generators in degree $0$, and not only one like the operads $\mathfrak{S}$ and $\mathcal{R}$. 
In the section \ref{sec-construction-operad-K}, we will construct an $E_\infty$-operad
that we denote $\mathcal{K}$. The construction is made by infinitely many steps, in the
sense that we construct a sequence of operads $\{\mathcal{K}_n\}_{n\geq 2}$, in 
such a way that $\mathcal{K}_i$ is a suboperad of $\mathcal{K}_{i+1}$. The operads
$\mathcal{K}_i$ are not $E_\infty$-operads, but they will be almost $E_\infty$-operads,
in the sense that until arity $i$ they will satisfy the $E_\infty$-conditions. Finally, the 
$E_\infty$-operad $\mathcal{K}$ is obtained by taking the inductive limit of this sequence of operads.

One of the characteristics of this construction is the use of a technique that
we call polynomial operads. It will create a new operad from an $\mathbb{S}$-module
containing an $\mathbb{S}$-submodule with an operadic structure, in such a way that this
operadic structure is preserved in the resulting operad. This done by using 
amalgamated sums in the category of operads. The section \ref{sec-polynomial-operads}
is completely dedicated to the description of this technique. 

In the final part we exhibit the main element $A[1]$ of an $\mathcal{L}$-algebra
as a $\mathcal{E}_\infty$-coalgebra. Again, this will be possible due to the sequence of 
operads that define $\mathcal{K}$, in the sense that it will be sufficient to exhibit
$A[1]$ as a $\mathcal{K}_i$-coalgebra for each $i$, because the universal property
of colimits will induce the $\mathcal{K}$-coalgebra structure on $A[1]$. 
Moreover, our construction that $A[1]$ is a $E_\infty$-coalgebra is functorial. 
This proves that an $\mathcal{L}$-algebra quasi-isomorphic to $\mathcal{A}(X)$
contains at least as many homotopy information as a $E_\infty$-coalgebra structure
on $C_*(X)$, such as the one described by Smith.

In \cite{Mandell2006CHT}, Mandell describes an $E_\infty$-algebra structure
on the normalized cochain complex associated to a simplicial space,
which under some finiteness hypothesis gives an invariant for the 
weak homotopy type of the space. Our results suggest that $\mathcal{L}$-algebras
are also pertinent in order to describe the weak homotopy type of spaces.

\section{Preliminaries}
\label{preliminaries}

	Let $\Lambda$ be a commutative ring with unity,
	a differential graded module over $\Lambda$ or simply $\dg$ module is
	a graded $\Lambda$-module $M$ together with 
	a homogeneous morphism $\partial:M\to M$ of degree $-1$, 
	called differential, such that $\partial^{2}=0$.
	When provided with homogeneous morphisms 
	$\epsilon:M\to 	\field$ and $\eta:\field \to M$, 
	called augmentation and coaugmentation, respectively,
	such that $\epsilon \circ \eta =\text{id}$, $M$ 
	is said to be a differential graded module with augmentation or $\dga$ module.
	The category of $\dg$ modules and $\dga$ modules, both with
	homogeneous morphisms, is denoted $\dgcat$ and $\dgacat$, respectively.
	Signs in this graded context follows Koszul convention (cf. \cite{loday2012algebraic}).

	$M$ is acyclic if its augmentation $\epsilon:M\to \Lambda$ induces an isomorphism in homology,
	and $M$ is said to be contractible if $\epsilon$ is a homotopy equivalence
	(cf. \cite{Alain-Ainf}).	
	A contracting chain homotopy of $M$, is a degree $1$ morphism $h:M\to M$
	homotopy from $0$ to $\text{Id}_M$. In other words, $h$ satisfies,
	$\partial h+h\partial=1$.
	If $H_*(M)=0$, $M$ is said to be null-homotopic.
	A morphism $f:M\to N$ of degree $k$ is said null-homotopic if 
	it is homotopic to $0$, that is, there is a morphism
	$h:M\to N$ of degree $k+1$ such that $\partial(h)=\partial h-(-1)^kh\partial=f$.
		
	When $M$ has a contracting chain homotopy then $M$ is null-homotopic,
	and as a converse, if $M$ projective, null-homotopic and bounded below, 
	then it has a contracting chain homotopy,
	which is constructed inductively.
	A similar proof lead to the following result.


	\begin{lemma}
		\label{lemma-relative-homotopic-zero-morphism}
		Let $f:L\to N$ be a homogeneous morphism of 
		$\dg$ modules
		of degree $k\in \mathbb{Z}$.
		Let $L'$ be a 
		$\dg$ submodule
		of $L$ and  $P$ graded submodule of $L$ such that 
		$P$ is projective and bounded below, 
		and in each degree we have the decomposition $L_i=L_i'\oplus P_i$.
		Write $f'$ for the restriction of $f$ to $L'$. 
		Suppose that the homology of $N$ is zero 
		and that $f'$ is null-homotopic 
		by a homotopy $h:L'\to N$. Then there exists a homotopy 
		$H$ on $L$ extending $h$, which makes $f$ null-homotopic.
	\end{lemma}
\qed

	\begin{theorem}[Relative Lifting Theorem]
		\label{th-relative-relevement}
		Let $f:M\to N$ and $\varphi:L\to N$ morphisms of $\dg$ modules 
		of degree $l$ and $k$, respectively.
		Suppose that $f$ is a quasi-isomorphism, and 
		let $L'$ be a $\dg$ submodule of $L$ and  
		$P$ graded submodule of $L$ such that 
		$P$ is projective and bounded below, 
		and in each degree we have the decomposition $L_i=L_i'\oplus P_i$.
		Write $\varphi'$ for the restriction of $\varphi$ to $L'$. 
		Suppose there is 
		a morphism $\alpha':L'\to M$ of $\dg$ modules of degree $k-l$ 
		that lifts $\varphi'$ up to homotopy along $f$.
		Then there exists an extension $\alpha:L\to M$ of $\alpha'$, 
		that lifts $\varphi$ up to homotopy along $f$.
		Moreover, the homotopy can be choose to be an extension of the homotopy associated to $\alpha'$.
			\begin{equation}
				\begin{gathered}
				\xymatrix{
				M\ar[rr]^-f & & N \\
				 & L  \ar@{-->}[ul]_-{\alpha}  \ar[ur]^-{\varphi}& \\
				 & L' \ar@/^1pc/[uul]^-{\alpha'} \ar@{^{(}->}[u] \ar@/_1pc/[uur]_-{\varphi'}& 
				}
				\end{gathered}
			\end{equation}
	\end{theorem}

	\begin{proof}
		Let $C(f)$ be the mapping cone of $f$. Let $u:N\to C(f)$ 
		the inclusion $x\mapsto \begin{pmatrix}0\\x\end{pmatrix}$
		and $h':L'\to N$ the homotopy from $\varphi'$ to $f\circ \alpha'$. Then we can easily check
		that $\begin{pmatrix}\alpha'\\ h' \end{pmatrix}:L'\to C(f)$ 
		is a homotopy to $0$ of $u\circ \varphi'=\begin{pmatrix}
		0\\ \varphi' \end{pmatrix}$. 
		Lemma \ref{lemma-relative-homotopic-zero-morphism} says there exists a homotopy to zero
		$\begin{pmatrix}H_1\\ H_2\end{pmatrix}:L\to C(f)$ of $u\circ \varphi$ extending  
		$\begin{pmatrix}\alpha'\\ h' \end{pmatrix}$. So we have,

		\begin{equation}
		\begin{gathered}
		\begin{array}{lcl}
		\begin{pmatrix}0\\ \varphi \end{pmatrix} &=& \partial_{C(f)}\begin{pmatrix}H_1\\ 
		H_2\end{pmatrix}
		+(-1)^{k}\begin{pmatrix}H_1\\ H_2\end{pmatrix}\partial_L\\
		&=&\begin{pmatrix}-(-1)^{l}\partial_M H_1  +(-1)^kH_1 \partial_L\\ 
		fH_1+\partial_N H_2 +(-1)^kH_2 \partial_L\end{pmatrix}\\
		\end{array}
		\end{gathered}
		\end{equation}

		This gives the following equations.

		\begin{equation}
		\begin{gathered}
		\begin{array}{lcl}
		\partial_M H_1 &=& (-1)^{l+k}H_1\partial_L\\
		\varphi-fH_1 &=& \partial_N H_2+(-1)^kH_2 \partial_L
		\end{array}
		\end{gathered}
		\end{equation}

		The first says that $H_1$ is a morphism of $\dg$-modules and the second that $H_2$ is a 
		homotopy from $f H_1$ to $\varphi$. Finally, we take $\alpha=H_1$ as the lift of $\varphi$ along $f$.
	\end{proof}

	Let $\field$ be a field.
	For $n$ positive integer, $\Sigma_n$ denote the symmetric group of $n$ elements
	and $\field[\Sigma_n]$ its group ring over $\field$ (cf. \cite{masson2008introduction}, \cite{weibel1995introduction}).


	\begin{proposition}[Acyclic extension]
		\label{prop-acyclic-complete}
		Let $M$ be a $\field[\Sigma_n]$-free finitely generated $\dga$ module over $\field$. Then
		there exists a $\field[\Sigma_n]$-free finitely generated acyclic $\dga$ module $N$
		over $\field$, 
		with $M_0=N_0$ and $M$ as $\dga$ submodule.
	\end{proposition}

	\begin{proof}

		For the modules on $\field[\Sigma_n]$, we consider the adjunction 
		$L\vdash U:\set \to \text{Mod}_{\field[\Sigma_n]}$,
		where $U$ is the forgetful functor.
		For every module $N$, the counit gives the surjection $\epsilon_N:LU(N)\to N$, which will be
		denoted $p:PN\to N$. Given a $\dga$ module $M$ we denote $ZM$ its submodule of cycles. On
		$ZM$ the differential is $0$, 
		then we extend the meaning of $P$ to graded modules, we keep the same notation
		for the extended morphism $p:PZM\to ZM$. Consider the composition $d=i\circ p:PZM\to M$, where
		$i$ is the canonical inclusion of $ZM$ in $M$. $ZM$ can seen as a submodule of $PZM$,
		then $p:PZM\to ZM$ is a retraction for this inclusion and if $m\in ZM$, then $d(m)=m$.

		Observe that in the mapping cone of $d:i\circ p:PZM\to M$, $C(d)$, 
		all the cycles of $M$ are now boundaries
		and also on $C(d)$ will appear new cycles. Indeed, let $m\in M$ cycle, 
		recall that the differential of $C(d)$
		is given by $\begin{pmatrix} 0 & 0 \\ d & \partial_M \end{pmatrix}$.
		Then in $C(d)$ we have $\partial \begin{pmatrix}m\\ 0 \end{pmatrix}= \begin{pmatrix} 0\\ m\end{pmatrix}$,
		which means that $m$ is a boundary.  If it happens that $m$ is already a boundary in $M$, that is
		there exists $n$ such that $\partial_M(n)=m$, then $\begin{pmatrix} m \\ -n\end{pmatrix}$
		is a cycle in $C(d)$. From this also notice that if all the cycles of $M$ have degree at least $k$,
		then all the cycles in $C(d)$ will have at least degree $k+1$.

		Let $M$ $\field[\Sigma_n]$-free finitely generated $\dga$ module, and denote $W$ the kernel
		of the augmentation $\epsilon:M\to \field$ and consider the $\field[\Sigma_n]$-linear
		morphisms for $n\geq 1$,

		\begin{equation}
		\begin{gathered}
		\xymatrix{
		PZC(d_n) \ar[r]^-{d_{n+1}} & C(d_n)
		}
		\end{gathered}
		\end{equation}

		where $d_1$ is $d:PZW\to W$, and $d_{n+1}$ is $d:PZC(d_n)\to C(d_n)$. Then we have
		that $W$ is included in $C(d_1)$ and  $C(d_{n})$ is included in $C(d_{n+1})$.
		With this we can define a $\dga$-module $N$ that satisfies the conditions of the theorem
		by taking the colimit of the following diagram,

		\begin{equation}
		\begin{gathered}
		\xymatrix{
		M & W \ar[l] \ar[r]& C(d_1) \ar[r]& C(d_2) \ar[r]& \cdots
		}
		\end{gathered}
		\end{equation}

		where all the arrows are the respective canonical inclusions. Observe that we can
		reduce the size of this acyclic extension by considering in the first step
		only the cycles of degree $0$ of $W$, and for the construction of $d_{n+1}$,
		considering only the cycles of degree $n$ of the last mapping cone.
	\end{proof}


\section{Operads}
\label{section-operads}

	In the following we consider differential graded modules over a field $\field$.
	An operad $P$ in the monoidal category $\dgacat$ 
	is a collection of $\dga$-modules $\{P(n)\}_{n\geq 1}$ together with 
	right actions of the symmetric group $\Sigma_n$ on each component $P(n)$, 
	and morphisms of the form $\gamma:P(r)\otimes P(i_1)\otimes P(i_r)\to P(i_1+\cdots+i_r)$,
	which satisfy the usual conditions of existence of an unit, associativity and equivariance.
	Morphisms $\gamma$ will be called composition morphisms of the operad.
	A morphism between operads $f:P\to Q$, is 
	a collection of $\dga$-morphisms $f_n:P(n)\to Q(n)$ of degree $0$,
	respecting units, composition and equivariance. 
	The category of operads is denoted $\mathcal{OP}$ (cf.  \cite{may2007}, \cite{markl2007operads}, \cite{loday2012algebraic}).

	Two fundamental examples of operads are the endomorphism operad
	and the coendomorphism operad, because their behavior inspired the definition of operad.
	For $M\in\dgacat$, the operad $\text{End}(M)$ of endomorphisms of $M$ is defined by
	taking $\text{End}(M)(n)=Hom(M^{\otimes n},M)$, 
	with unit $\eta:\field \to \text{End}(M)(1)$ given by $\eta(1)=\text{Id}_{M}$,
	right action of $\Sigma_n$ over $\text{End}(M)$ induced by the left
	action of $\Sigma_n$ over $M^{\otimes n}$,
	and obvious composition applications.
	For the coendomorphism operad $\text{Coend}(M)$ we take
	$\text{Coend}(N)(n)=Hom(N,N^{\otimes n})$ with unit 
	$\eta(1)=\text{Id}_{N}$,
	right action of $\Sigma_n$ over $\text{Coend}(N)$
	is the induced by the right action of
	$\Sigma_n$ over $N^{\otimes n}$ and obvious compositions.

	An important feature in the theory of operads are its representations.
	That is, when the abstract operations of the operads are
	interpreted as concrete application over an object in the ground
	category.
	This passage from the abstract to the concrete is made
	through morphisms of type $P(n)\to \text{Hom}(A^{\otimes n},A)$. In this sens, an
	element of $\mathcal{P}$ is 
	realized as an $n$-ary operation over $A$. 
	This association must be coherent with composition
	operations and  symmetric groups actions.

	An algebra over the operad $\mathcal{P}$, or $\mathcal{P}$-algebra, 
	is a $\dga$ module $A$, together with a morphism of operads from $\mathcal{P}$ to $\text{End}(A)$.
	Similarly, a coalgebra is a $\dga$ module $C$,
	together with a morphism of operad from $\mathcal{P}$ to $\text{Coend}(C)$.

	In the symmetric monoidal category $\dgacat$, 
	for every $Y$ $\dga$ module 
	the functor $-\otimes Y$ is left adjunct of the functor $\text{Hom}(Y,-)$.
	Denote $\theta$ the natural bijection
	$\theta_{X,Z}: \text{Hom}(X\otimes Y, Z)\to \text{Hom}(X, \text{Hom}(Y,Z))$  given by this adjunction.
	Then, for a morphism of operads $f:\mathcal{P}\to \text{End}(A)$,  
	each component $f_n:P(n)\to \text{Hom}(A^{\otimes n},A)$ determines a morphism of $\dga$ modules
	$\varphi_n:P(n)\otimes A\to A^{\otimes n}$, given by $\varphi_n=\theta^{-1}(f_n)$.
	This allows defining $P$-algebras and $P$-coalgebras
	equivalently by a collection $\{\varphi_n\}_{n\geq 1}$ 
	of morphisms of $\dga$ modules
	$\varphi_n: P(n)\otimes C^{\otimes n} \to C$,
	and
	$\varphi_n: P(n)\otimes C \to C^{\otimes n}$,
	respectively,
	which satisfying the usual conditions of associativity, unit and equivariance.

	If we forget composition morphisms of an operad $P$, 
	the collection of $\dga$ modules with right actions that remains
	is called an $\mathbb{S}$-module.
	They form a category denoted $\smodcat$,
	which has all colimits and limits because
	it is a category of diagrams over $\dgacat$.

	The forgetful functor $U:\mathcal{OP}\to \smodcat$
	has a right adjoint denoted $F:\smodcat\to \mathcal{OP}$, 
	called the free operad functor
	(cf. \cite{rezkthesis1996},\cite{markl2007operads}). 
	This adjunction 
	defines for every operad $\mathcal{Q}$ and $\mathbb{S}$-module $M$, the natural bijection,

	\begin{equation}
	\begin{gathered}
	\theta: \mathcal{OP}(F(M),\mathcal{Q})\to \mathbb{S}\text{-Mod}(M,U(\mathcal{Q}))
	\end{gathered}
	\end{equation}

	With unit and counit denoted by $\eta$ and $\epsilon$,
	respectively. 
	Recall that $\eta_M:M\to UF(M)$ and $\epsilon_\mathcal{P}:FU(\mathcal{P})\to \mathcal{P}$.

	The category of operads has all small colimits 
	(cf. \cite{fresse-homoperad-2016}), which will be used 
	to construct the $E_\infty$-operad $\mathcal{K}$
	in section \ref{section-principal-results}.

\subsection{$E_\infty$-Operads}
\label{subsection-e-infty-operads}

\begin{definition}[$E_\infty$-Operad] 
\label{df-Einf-Operad}
	An operad $\mathcal{P}$ on the category $\dgacat$ is called $E_\infty$-operad 
	if each component $P(n)$ is a $\field[\Sigma_n]$-free resolution of $\field$.
\end{definition}

\begin{definition}[$E_\infty$-algebra and $E_\infty$-coalgebra] 
\label{df-Einf-Algebra-Coalgebra}
	We call $E_\infty$-algebra any $\mathcal{P}$-algebra with $\mathcal{P}$ an
	$E_\infty$-operad. And in the same way, an $E_\infty$-coalgebra
	is an $\mathcal{P}$-coalgebra where the operad $\mathcal{P}$ is an $E_\infty$-operad.
\end{definition}

We introduce a notion of morphism between $E_\infty$-coalgebras which is well suited for our purpose.

\begin{definition}
	Let $\mathcal{P}$ be an $E_\infty$-operad on the category $\dgacat$, and let $A,B$
	$\mathcal{P}$-coalgebras. A morphism $f:A\to B$ of $\mathcal{P}$-coalgebras
	is a morphism of $\dgacat$ which preserves the $\mathcal{P}$-coalgebra
	structure up to homotopy, that is, the following diagram 

		\begin{equation}
		\begin{gathered}
		\xymatrix{
		\mathcal{P}(n)\otimes A\ar[r]^-{\varphi_n^A} \ar[d]_-{1\otimes f}& A^{\otimes n} \ar[d]^-{f^{\otimes n}}\\
		\mathcal{P}(n)\otimes B \ar[r]_-{\varphi_n^B} & B^{\otimes n}
		}
		\end{gathered}
		\end{equation}

	is commutative up to homotopy for every $n>0$, where $\varphi^A_n$ and $\varphi^B_n$
	are the associated morphisms of the $\mathcal{P}$-coalgebra structure of $A$ and $B$,
	respectively. The category of coalgebras on the operad $\mathcal{P}$ is denoted
	$\mathcal{P}$-\text{CoAlg}.
\end{definition}

\subsection{Polynomial Operads}\label{sec-polynomial-operads}
\label{subsection-polynomial-operads}

The polynomial operads construction is a technique used to create an operad from 
an $\mathbb{S}$-module with an $\mathbb{S}$-submodule having an operadic structure, 
in such a way that this operadic structure is preserved. 
Recall that we denote by $U$ the forgetful functor from operads to $\mathbb{S}$-modules.

\begin{definition}
	\label{df-category-C-operad-S-mod}
	$\mathfrak{C}$ is the category such that,
		\begin{enumerate}

		\item The objects are pairs of the form $(\mathcal{E},M)$, where $M$ is a $\mathbb{S}$-module and
		$\mathcal{E}$ is an operad such that $U(\mathcal{E})$ is a $\mathbb{S}$-submodule
		of $M$. The canonical inclusion is denoted by $i_\mathcal{E}:U(\mathcal{E})\to M$.

		\item A morphism from $(\mathcal{E},M)$ to $(\mathcal{F},N)$, is 
		a pair $(f,\overline{f})$ with $f:\mathcal{E}\to \mathcal{F}$ morphism of operads,
		and $\overline{f}:M\to N$ morphism of $\mathbb{S}$-modules, 
		such that the following diagram commutes.

		\begin{equation}\label{dgm-morphism-cat-C}
		\begin{gathered}
		\xymatrix{
		M \ar[r]^-{\overline{f}} & N\\
		U(\mathcal{E}) \ar[u]^-{i_{\mathcal{E}}} \ar[r]_-{U(f)}& U(\mathcal{F}) \ar[u]_-{i_{\mathcal{F}}}
		}
		\end{gathered}
		\end{equation}

		\end{enumerate}
\end{definition}

	Essentially, a morphism from $(\mathcal{E},M)$ to $(\mathcal{F},N)$ in $\mathfrak{C}$ is morphism
	of $\mathbb{S}$-modules
	from $M$ to $N$ that sends $U(\mathcal{E})$
	to $U(\mathcal{F})$ and respects the operadic structure of $\mathcal{E}$.

\begin{definition}
\label{df-forgetful-OP-to-C}
	We define $\mathfrak{U}:\mathcal{OP}\to \mathfrak{C}$ to be the functor forgetful
	which sends every operad $\mathcal{E}$ to the pair $(\mathcal{E},U(\mathcal{E}))$.
	That is, every operad is sent to the pair formed by itself and its underlying $\mathbb{S}$-module.
\end{definition}

\begin{theorem}
\label{th-adj-poly-operads}
	The functor $\mathfrak{U}:\mathcal{OP}\to \mathfrak{C}$ has a left adjoint.
	We denote this adjoint by $\mathfrak{J}$, 
	and the image of $(\mathcal{E},M)$ under $\mathfrak{J}$ by $\mathcal{E}[M]$
	and called the polynomial operad on $M$ with coefficients in $\mathcal{E}$.
\end{theorem}

\begin{proof}
	We can associate to every $(\mathcal{E},M)\in \mathfrak{C}$ the following diagram in $\mathcal{OP}$,

	\begin{equation}\label{diagram-cognomon-op}
	\begin{gathered}
	\xymatrix{
	FU(\mathcal{E}) \ar[r]^-{\epsilon_{\mathcal{E}}} \ar[d]_-{F(i_{\mathcal{E}})} & \mathcal{E} \\
	F(M) &\\
	}
	\end{gathered}
	\end{equation}

	where $\epsilon:FU\to 1_{\mathcal{OP}}$ is the counit of the adjunction 
	$F\vdash U:\mathbb{S}\to \mathcal{OP}$. 
	This association is functorial by
	the naturality of the counit $\epsilon$ and the definition
	of $(f,\overline{f})$ as a morphism in $\mathfrak{C}$.
	Thus we have a functor $Cm$ from $\mathfrak{C}$ to 
	the category of diagrams in $\mathcal{OP}$ of the form
	$\xymatrix{{\bullet}  & {\bullet} \ar[r] \ar[l]& {\bullet}}$
	Then, we define the functor $\mathfrak{J}:\mathfrak{C}\to \mathcal{OP}$
	to be the composition of $Cm$ with the functor of colimits
	on $\mathcal{OP}$.

	In order to prove that we have the adjunction 
	$\mathfrak{J}\vdash \mathfrak{U}:\mathfrak{C}\to \mathcal{OP}$,
	we only have to
	construct for every object $(\mathcal{E},M)\in \mathfrak{C}$
	an universal arrow $\Psi$ from $(\mathcal{E},M)$ to  
	$\mathfrak{U}\mathfrak{J}(\mathcal{E},M)=(\mathcal{E}[M],U(\mathcal{E}[M]))$.

	Let $(\mathcal{E},M)$ be an object in $\mathfrak{C}$ and consider the following
	diagram given by the colimit $\mathfrak{J}(\mathcal{E},M)$.

	\begin{equation}\label{diagram-EM-colimit}
	\begin{gathered}
	\xymatrix{
	 FU(\mathcal{E}) \ar[r]^-{\epsilon_{\mathcal{E}}} \ar[d]_-{F(i_{\mathcal{E}})} & \mathcal{E} \\
	F(M) &  \mathcal{E}[M] \ar@{<-}[u]_-{\alpha} \ar@{<-}[l]^-{\beta} \\
	}
	\end{gathered}
	\end{equation}

	Now consider the couple of arrows $(\alpha,\theta(\beta))$, where
	$\theta$ is the isomorphism, 

	\begin{equation}
	\begin{gathered}\label{diagram-iso-F-U-adjuntion}
	\xymatrix{
	\mathcal{OP}(F(M),P) \ar[r]^-{\theta} & \mathbb{S}(M,U(P))
	}
	\end{gathered}
	\end{equation}

	given by the adjunction $F\vdash U$. 
	This couple will be our universal arrow $\Psi$. 
	Checking that $\Psi$
	is a morphism in $\mathfrak{C}$ and that it satisfies the universal property,
	is routine.
\end{proof}

	The universal arrow 
	in the proof of theorem \ref{th-adj-poly-operads},
	extends to the unit of the adjunction
	$\mathfrak{J}\dashv \mathfrak{U}:\mathcal{C}\to \mathcal{OP}$.
	We keep the notation $\Psi$ for this unit. 
	The universal property for the unit $\Psi:1_{\mathfrak{C}}\to \mathfrak{U}\mathfrak{J}$,
	gives the following result.

	\begin{proposition}
	\label{prop-property-unit-adjunction-poly-op}
	Let $(\mathcal{E},M)\in \mathfrak{C}$ and $\mathcal{A}\in \mathcal{OP}$.
	For every morphism 
	$(f,\overline{f}):(\mathcal{E},M)\to \mathfrak{U}(\mathcal{A})=(\mathcal{A},U(\mathcal{A}))$,
	there exists an unique morphism of operads $\varphi:\mathcal{E}[M]\to \mathcal{A}$,
	such that $U(\varphi)\Psi = \overline{f}$. So we have the following commutative diagram.

	\begin{equation}
	\begin{gathered}
	\xymatrix@C=4pc{
	(\mathcal{E},M) \ar[r]^-\Psi \ar[rd]_-{(f,\overline{f})} & (\mathcal{E}[M],U(\mathcal{E}[M])) \ar[d]^-{(\varphi,U(\varphi))}\\
	& (\mathcal{A},U(\mathcal{A}))
	}
	\end{gathered}
	\end{equation}

	\end{proposition}

\qed



		In the following we proof some properties of polynomial operads construction 
		which will be used in defining operad $\mathcal{K}$ in section 
		\ref{section-principal-results}.


		\begin{definition}
		\label{df-theta-n-cut-s-mod}
		Let $\theta_n:\mathbb{S}\text{-Mod}\to \mathbb{S}\text{-Mod}$ be the functor
		which sends each $\mathbb{S}$-module $M$ to the $\mathbb{S}$-module $\theta_n(M)$ such
		that $\theta_n(M)(k)=M(k)$ if $k\leq n$ and $\theta_n(M)(k)=0$ if $k>n$.
		The $\mathbb{S}$-modules in the image of the functor $\theta_n$ are
		called $n$-$\mathbb{S}$-modules.
		\end{definition}

		\begin{proposition}
		\label{prop-n-Mod-M-N}
		For every pair of $\mathbb{S}$-modules $M$, $N$ we have
		\begin{equation}\label{diag-n-Mod-M-N}
		\begin{gathered}
		\theta_n(M\circ N)=\theta_n(\theta_n(M)\circ \theta_n(N))
		\end{gathered}
		\end{equation}
		\end{proposition}

		\begin{proof}
		Observe that in a composition the arities of the operations
		composed have lower arities than the resulting operations,
		because we only consider operations with arities $\geq 1$. 
		\end{proof}

		\begin{definition}
		\label{Def-cat-OPn}
		Let $\mathcal{OP}_n$ be the subcategory of the category of operads which are null for arities $>n$.
		We call the objects of $\mathcal{OP}_n$, $n$-operads.
		$T_n:\mathcal{OP}\to \mathcal{OP}_n$ is the functor which assigns to each operad the
		quotient $n$-operad determined by the operad ideal of operations of arities $>n$.
		In other words, the resulting operad has null compositions when the resulting operation
		has arity $>n$. $I:\mathcal{OP}_n\to \mathcal{OP}$ is the inclusion functor
		and satisfies $T_n\circ I = 1_{\mathcal{OP}_n}$, for every $n>1$.
		\end{definition}

		Next proposition says that in order to construct the part of $F(M)$ with arity at most $n$,
		is enough to consider operations of $M$ with arity at most $n$. 

		\begin{proposition}
		\label{prop-TFFTN}
		For $n>1$, $T_n\circ F=T_n\circ F\circ \theta_n$
		\end{proposition}

		\begin{proof}
		By induction over the construction of $F$. $\theta_nF_{k+1}(M)=\theta_n(I\oplus (M\circ F_k(M)))$,
		which is equal to $\theta_n(I\oplus \theta_n(M\circ F_k(M)))$, by proposition \ref{prop-n-Mod-M-N}
		this is $\theta_n(I\oplus (\theta_n(M)\circ \theta_n(F_k(M)))$ and by induction
		hypothesis we obtain $\theta_nF_{k+1}(M)=\theta_n(I\oplus (\theta_n(M)\circ F_k(\theta_n(M))))$,
		and then $\theta_n$ commute with the colimit defining $F$ because it is filtering.
		Finally $\theta_n$ becomes $T_n$ when the operad structure is added.
		\end{proof}

		\begin{proposition}\label{prop-Tn-I-Adjunction}
		For every $n>1$ we have the adjunction $I\vdash T_n:\mathcal{OP}\to \mathcal{OP}_n$.
		\end{proposition}

		\begin{proof}
		We only have to check for every $n\text{-}\mathbb{S}$-module $Q$ that the identity 
		$1_Q:T_nI(Q)\to Q$ in $\mathcal{OP}_n$ is a universal
		arrow from the functor $T_n$ to the object $Q$.
		Observe that every morphism of the form $g:T_n(P)\to Q$ determines a morphism $\varphi:P\to I(Q)$
		by taking $\varphi(k)=0$ if $k>n$ and equal to $g$ in lower degrees.

		\begin{equation}
		\begin{gathered}
		\xymatrix{
		T_nI(Q)\ar[rr]^-1 && Q\\
		&T_n(P) \ar[ul]^-{T_n(Q)} \ar[ur]_-{g}&
		}
		\end{gathered}
		\end{equation}

		\end{proof}

		\begin{definition}
		An $n$-isomorphism $f:P\to Q$ of operads is an morphism of operads such that 
		$T_n(f)$ is an isomorphism of $n$-operads.
		\end{definition}

		\begin{lemma}
		\label{lemma-alpha-iso}
		In the following diagram from the definition of polynomial operads,
		\begin{equation}
		\label{prop-diagram-pol-op}
		\begin{gathered}
		\xymatrix{
		FU(\mathcal{E})\ar[r]^-\epsilon \ar[d]_-{F(i)} & \mathcal{E} \ar[d]^-\alpha \\
		F(M) \ar[r]_-{\beta} & \mathcal{E}[M]
		}
		\end{gathered}
		\end{equation}
		if $i$ is an $n$-isomorphism (of $\mathbb{S}$-modules) then the operad morphism $\alpha:\mathcal{E}\to \mathcal{E}[M]$ is an $n$-isomorphism.

		\end{lemma}

		\begin{proof}

		By \ref{prop-TFFTN} $(T_n\circ F)(i)=(T_n\circ F\circ \theta_n)(i)$, but $\theta_n(i)$ is an
		isomorphism, so that $T_nF(i)$ is also an isomorphism, which means that $F(i)$ is an $n$-isomorphism.
		By \ref{prop-Tn-I-Adjunction} the following diagram is cocartesian,

		\begin{equation}\label{prop-diagram-pol-op}
		\begin{gathered}
		\xymatrix{
		T_nFU(\mathcal{E})\ar[r]^-{T_n(\epsilon)} \ar[d]_-{T_nF(i)} & T_n\mathcal{E} \ar[d]^-{T_n(\alpha)} \\
		T_nF(M) \ar[r]_-{T_n(\beta)} & T_n\mathcal{E}[M]
		}
		\end{gathered}
		\end{equation}

		so that $T_n(\alpha)$ is an isomorphism, which means that $\alpha$ is an $n$-isomorphism.
		\end{proof}


\section{$\mathcal{L}$-Algebras}
\label{section-l-algebras}

	Associated to chain complexes of DGA modules, there are applications like the Eilenberg-Maclane 
	transformation and the Alexander-Whitney diagonal, both gives important information about
	the homotopy type of spaces (cf. \cite{may2007}). Under certain conditions the Alexander-Whitney diagonal
	can be described by unicity (cf. \cite{alain-transf-EM-1983}, \cite{alain-transf-AW-1984}).
	Following these ideas, with $\mathcal{L}$-algebras Alain Prouté associated to each space a structure
	with abstract properties inspired in such of chain complexes that allow unicity of Alexander-Whitney diagonal.

\subsection{The Category $\mathcal{L}$}

	With $\mathcal{L}$-algebras our principal interest is to model
	the relations describing the behavior of diagonals in chain complexes.
	This can be done by using an approach similar to simplicial objects, 
	that is, defining $\mathcal{L}$-algebras as contravariant functors from a suitable category. 
	This category will be denoted $\mathcal{L}$. 

	The objects of the category used by G. Segal, $\Gamma$ (see \cite{SEGAL1974293}) 
	are the finite sets, 
	and a morphism from $x$ to $y$ is an application $f:x\to \mathcal{P}(y)$ 
	\footnote[2]{$\mathcal{P}(y)$ is the set of subsets of $y$.}
	such that  $z_1\neq z_2$ implies $f(z_1)\cap f(z_2)=\emptyset$.
	Then we have the isomorphisms of categories 
	$\Gamma^{\text{op}}\cong \mathcal{L}$ and
	$\mathcal{L}^{\text{op}}\cong \Gamma$.

		\begin{definition}
		\label{definition-category-L}
			We define $\mathcal{L}$ to be the category with objects 
			totally ordered sets $[n]=\{1,\ldots, n\}$ for $n>0$ and $[0]=\emptyset$, 
			the empty set. Morphisms in $\mathcal{L}$ are taken to be all 
			the partial maps between these sets.
		\end{definition}

	An morphism $\alpha:[n]\to [m]$ of $\mathcal{L}$ could be described by a pair 
	$(\text{Dom}(f),f)$, where $\text{Dom}(f)\subseteq [n]$ is the domain of $f$.

	When a cocartesian category has a zero object is called pointed cocartesian category. 
	Furthermore, if the zero and the sum are explicitly given, 
	the category is called strict pointed cocartesian category.
	$\mathcal{L}$ is a strict pointed cocartesian category,
	with the strictly associative sum $[n]+[m]:=[n+m]$, 
	and with zero object $[0]$.

	\begin{definition}
	\label{df-special-arrows-L}
	In $\mathcal{L}$ we identify the following arrows.

		\begin{enumerate}
		\item The face operator $d_i:[n]\to [n+1]$ is defined for $1\leq i \leq n+1$ by:
		\begin{equation}
		\begin{gathered}
		d_i(x)=\left\{ 
		\begin{array}{l l}
		  x & \quad \text{if $x< i$}\\
		  x+1 & \quad \text{if $x\geq i$}\\ \end{array} \right. 
		\end{gathered}
		\end{equation}

		When $n=0$, the only face operator $d_1:[0]\to [1]$ is the universal morphism from $[0]$.

		\item The degeneracy operator $s_i:[n]\to [n-1]$ is defined for $1\leq i \leq n$ by:

		\begin{equation}
		\begin{gathered}
		s_i(x)=\left\{ 
		\begin{array}{l l}
		  x & \quad \text{if $x\leq i$}\\
		  x-1 & \quad \text{if $x> i$}\\ \end{array} \right. 
		\end{gathered}
		\end{equation}

		In the case $n=1$, the only degeneracy operator $s_1:[1]\to [0]$ is the universal morphism to $[0]$.

		\item In $\mathcal{L}$, any injective map $i:[n]\to [m]$ of the form $([n],i)$
		has a unique minimal retraction,
		denoted by $\overline{i}:[m]\to[n]$, in other words, $\overline{i}$ is the 
		only morphism with domain given by the image of $i$ and which satisfies the relation 
		$\overline{i}\circ i=1_{[n]}$.
		In particular, the minimal retraction associated to the face operator $d_i$ will be denoted $\zeta_i$.
		For $d_1:[0]\to[1]$, its minimal retraction $\zeta_1:[1]\to [0]$ coincide with $s_1:[1]\to [0]$.

		\end{enumerate}
	\end{definition}

	All morphisms of $\mathcal{L}$ can be generated by sum and compositions of the following five arrows.

	\begin{equation}
	\begin{gathered}
	\xymatrix@C=4pc{
	[0] \ar@/^1pc/[r]^-{d_1} & [1] \ar@(r,u)[]_1 \ar@/^1pc/[l]^-{\zeta_1}& [2] \ar@(r,u)[]_\tau \ar[l]^-{s_1}
	}
	\end{gathered}
	\end{equation}

	Category $\mathcal{L}$ can be characterized  as the
	free strictly associative pointed cocartesian category on one object,
	as the following proposition shows.

	\begin{proposition}
	\label{free-cocartesian-category}   
	Let $\mathcal{C}$ be a strictly associative pointed cocartesian category, and $X$ an
	object of $\mathcal{C}$. Then there an unique functor $F:\mathcal{L}\to \mathcal{C}$
	preserving zero and coproducts and such that $F([1])=X$.
	\end{proposition}

	\begin{proof}
	Indeed, $F([n])$ must be the $n$-fold sum $X+\overset{(n)}{\cdots}+X$ and $F([0])$ must be the zero
	object of $\mathcal{C}$. The five morphisms above have mandatory images by $F$,
	this means that $F(1)=1_X$, $d_1$ and $\zeta_1$
	are send to the unique morphisms $0\to X$ and $X\to 0$, where $0$ is the zero object of $\mathcal{C}$,
	the image of $s_1$ is the codiagonal of $[1]$, 
	that is the morphisms $[1]+[1]\to [1]$ obtained by the universal property
	of coproduct, so its image by $F$ must be $X+X\to X$, the codiagonal of $X$,
	which is well defined because $\mathcal{C}$ is cocartesian, that is, the sum is well defined.
	And $\sigma$ which is the canonical twisting arrow of the sum $[1]+[1]$,
	should be send to the canonical twisting arrow of $X+X$.
	\end{proof}

	Similarly, the opposite category $\mathcal{L}^{\text{op}}$ of $\mathcal{L}$ is
	characterized as the free strictly associative pointed cartesian category 
	generated by $[1]$.

\subsection{$\mathcal{L}$-Algebras}
\label{subsection-l-algebras}


An $\mathcal{L}$-algebra is a contravariant functor 
from $\mathcal{L}$ to a category with a notion of homology,
together with a natural transformation $\mu$, which is called
the product of the $\mathcal{L}$-algebra. 
Homotopy coherence is reflected in requiring product $\mu$ induces isomorphisms in homology.
Then, $\mathcal{L}$-algebra are defined in categories equipped with quasi-isomorphisms,
that is, with a distinguished class of arrows, called quasi-isomorphisms,
which forms a subcategory. The only categories
of this kind we will use are $\dgacat$
and $\dga$-Alg, both over a field $\field$, where being an quasi-isomorphisms means inducing an
isomorphism in homology.

	\begin{definition}
	[$\mathcal{L}$-algebra]
	\label{df-L-algebra}
	Let $(\mathcal{C},\otimes,k,T)$ be a strict symmetric monoidal category with quasi-isomorphisms.
	An $\mathcal{L}$-algebra $A$ with values in the category $\mathcal{C}$ consists of a functor,

		\begin{align}
		A:\mathcal{L}^{\text{op}}\to \mathcal{C}
		\end{align}

	together with a natural transformation 
	$\mu:\otimes \circ (\mathcal{A}\times \mathcal{A})\to \mathcal{A}\circ +$.

	The morphism in $\mathcal{C}$ that $\mu$ associates to each pair $([n],[m])\in\mathcal{L}\times \mathcal{L}$, 
	goes from $A[n]\otimes A[m]$ to $A[n+m]$ and is written $\mu_{[n],[m]}$. The image of any arrow $\alpha$ 
	is simply written again as $\alpha$,
	but this image goes in the opposite direction.
	Then, for every pair of arrows $\alpha:[p]\to [n]$ and $\beta:[q]\to [m]$ in $\mathcal{L}$
	we have the following commutative diagram.

		\begin{equation}
		\begin{gathered}
		\xymatrix@C=6pc{
		A[n]\otimes A[m]\ar[r]^-{\mu_{[n],[m]}} & A[n+m]\\
		A[p]\otimes A[q]\ar[r]^-{\mu_{[p],[q]}} \ar@{<-}[u]^-{\alpha \otimes \beta} 
				& A[p+q] \ar@{<-}[u]_-{\alpha +\beta}
		}
		\end{gathered}
		\end{equation}

	The functor $\mathcal{A}$ and 
	the natural transformation $\mu$ 
	are required to satisfy the following conditions.

	\begin{enumerate}
	\item \textbf{Associativity:} $\mu \circ (\mu \otimes 1)=\mu \circ (1\otimes \mu)$. 

	\item \textbf{Commutativity:} Let $[n]$, $[m]$ in $\mathcal{L}$, and $\tau:[m+n]\to [n+m]$ be the 
	twisting morphism of $[m]+[n]$, then the following diagram commutes.

	\begin{equation}
	\begin{gathered}
	\xymatrix{
	A[n]\otimes A[m] \ar[d]_-T \ar[r]^-\mu & A[n+m]\ar[d]^-\tau \\
	A[m]\otimes A[n] \ar[r]^-\mu & A[m+n]
	}
	\end{gathered}
	\end{equation}

	\item \textbf{Unit:} The image of $[0]$ by $\mathcal{A}$ is $k$ and 
	$\mu_{[0],[n]}=\mu_{[n],[0]}=1$.

	\item \textbf{Coherence:} For every pair $([n],[m])$, the morphism 
	$ \mu_{[n],[m]}$ is a quasi-isomorphism.

	\end{enumerate}
	The natural transformation $\mu$ is called the product 
	or the structural quasi-isomorphism of $\mathcal{A}$. 
	Also, in order to simplify the expressions 
	we drop the indexes of $\mu_{[n],[m]}$ and simply write $\mu$ when necessary. 
	\end{definition}

	There is an degenerated case of $\mathcal{L}$-algebra
	when 
	$A[n]=A[1]^{\otimes n}$ for every $n\leq1$ and
	$\mu$ is taken to be the identity,
	which implies that
	$s_1:A[1]\to A[1]\otimes A[1]$ is a commutative coproduct .
	Indeed, in $\mathcal{L}$ we have the following commutative diagram.

	\vspace{-1.5cm}
	\begin{equation}
	\begin{gathered}
	\xymatrix{
	[1] & [2] \ar[l]_-{s_1} \ar[d]^-{\tau}\\
	 & [2] \ar[ul]^-{s_1}
	}
	\begin{matrix}
	{}&{}\\
	{}&{}\\
	{}&{}\\
	{}&{}\\
	{}&{}\\
	{}&\Longrightarrow \\
	{}&{}
	\end{matrix}
	\xymatrix{
	A[1] \ar[dr]_-{s_1} \ar[r]_-{s_1} & A[2] \ar[d]^-{\tau}& A[1]\otimes A[1] \ar[l]^-{\mu=1}\ar[d]^-{T} \\
	 & A[2] & A[1]\otimes A[1]  \ar[l]^-{\mu=1}
	}
	\end{gathered}
	\end{equation}

	Taking $T=\tau$ and $T\circ s_1=s_1:A[1]\to A[1]\otimes A[1]$. 
	The fact that the coproduct is commutative implies 
	that higher homotopies of diagonals are zero, 
	making this $\mathcal{L}$-algebras not very interesting for us. 
	The reason is, 
	$\mathcal{L}$-algebras are supposed to model the behavior of systems
	of diagonals like the one found in the chain complex
	associated to a simplicial set. In that case the diagonals obtained 
	from homotopy inverses of the Eilenberg-Mac~Lane transformation are not commutative, 
	because of the existence of Steenrod operations. 

	\begin{proposition}
	[The category of $\mathcal{L}$-algebras]
	\label{prop-def-L-algebras}
	$\mathcal{L}$-algebras with values in a category $\mathcal{C}$ 
	together with 
	natural transformations $f:\mathcal{A}\to \mathcal{B}$ 
	such that

		\begin{enumerate}
		\item $f$ preserves the product of $\mathcal{L}$-algebras,
		that is, for every pair $([n],[m])$ the following diagram is commutative.

		\begin{equation}
		\begin{gathered}
		\xymatrix@C=6pc{
		A[n]\otimes A[m] 
		\ar[r]^-{\mu_{\mathcal{A}}}  
		\ar[d]_-{f_{[n]}\otimes f_{[m]}} & 
		A[n+m] \ar[d]^-{f_{[n+m]}}\\
		B[n]\otimes B[m] \ar[r]^-{\mu_{\mathcal{B}}} & B[n+m]
		}
		\end{gathered}
		\end{equation}

		\item $f_{[0]}:A[0]\to B[0]$ is the identity of $k$.

		\end{enumerate}
	form a category. This category will be denoted $\mathcal{L}(\mathcal{C})$.
	\end{proposition}

	\begin{proof}
	Let $\mathcal{A}$, $\mathcal{B}$, $\mathcal{C}$ be three $\mathcal{L}$-algebras,
	$f:\mathcal{A}\to \mathcal{B}$ and $g:\mathcal{B}\to \mathcal{C}$,
	be two morphisms of $\mathcal{L}$-algebras. It suffice to check that
	the composition of natural transformation $g\circ f$ is a morphism
	of $\mathcal{L}$-algebras. Second condition is trivial and first condition 
	is consequence of the following commutative diagrams for $f$ and $g$.

	\begin{equation}
	\begin{gathered}
	\xymatrix@C=4pc{
	A[n]\otimes A[m] 
	\ar@/_4pc/[dd]_-{{(g\circ f)}_{[n]}\otimes {(g\circ f)}_{[m]}}
	\ar[r]^-{\mu_{\mathcal{A}}}  \ar[d]^-{f_{[n]}\otimes f_{[m]}} 
	& A[n+m] \ar@/^4pc/[dd]^-{{(g\circ f)}_{[n+m]}} \ar[d]^-{f_{[n+m]}}\\
	B[n]\otimes B[m] \ar[d]^-{g_{[n]}\otimes g_{[m]}} \ar[r]^-{\mu_{\mathcal{B}}} & B[n+m] \ar[d]^-{g_{[n+m]}}\\
	C[n]\otimes C[m] \ar[r]^-{\mu_{\mathcal{C}}} & C[n+m]
	}
	\end{gathered}
	\end{equation}
	\end{proof}


	For an $\mathcal{L}$-algebra $A$ with values in the category $\mathcal{C}$,
	if $\mathcal{C}$ is the category of $\dga$ algebras, 
	$A$ is called a multiplicative $\mathcal{L}$-algebra. 
	$\mathcal{L}$-algebras are designed to
	represent the $0$-reduced simplicial sets and multiplicative $\mathcal{L}$-algebras will 
	represent the $0$-reduced simplicial groups.

\subsection{The Monoidal Category of $\mathcal{L}(\mathcal{C})$}\label{sec-monoidal-struct-l-alg}

Let $T:\mathcal{L}^{\text{op}}\to \mathcal{C}$ be the functor defined by $T[n]=\field$
for every $n\geq 0$ and $T(\alpha)=1_\field$ for every morphism in $\mathcal{L}$. Together
with the natural transformation $\mu:\otimes \circ (T\times T)\to T\circ +$
defined by $\mu_{[n],[m]}=1_\field$ for all $[n],[m]\in \mathcal{L}$, 
the functor $T$ is an $\mathcal{L}$-algebra,
it is called the trivial $\mathcal{L}$-algebra with values in $\mathcal{C}$
and is denoted $\field$.

\begin{proposition}
Let $\mathcal{C}$ be the category $\dga$-modules. 
Then the trivial $\mathcal{L}$-algebra $\field$ is a zero object in $\mathcal{L}(\mathcal{C})$.
\end{proposition}

	\begin{proof}
	$\field$ is the zero object of $\dgacat$, then,
	for any $\mathcal{L}$-algebra $A$, this defines unique $\dga$-morphisms
	$i_{[n]}:\field \to A[n]$ and $p_{[n]}:A[n]\to \field$ ($n\geq 0$),
	which coincide with the coaugmentation and augmentation of $A[n]$, respectively.
	The associated natural transformations $i:\field \to A$ and $p:A\to \field$
	are morphisms of $\mathcal{L}$-algebras by commutativity of the following diagram.
		\begin{equation}
		\begin{gathered}
		\xymatrix@C=2pc{
		&\field \ar[dl]_-{\eta_n \otimes \eta_m} \ar[dr]^-{\eta_{n+m}}&\\
		A[n]\otimes A[m] \ar[rr]_-{\mu}  & & A[n+m] \\
		&\field \ar[ul]^-{\epsilon_n \otimes \epsilon_m} \ar[ur]_-{\epsilon_{n+m}}&\\
		}
		\end{gathered}
		\end{equation}
	\end{proof}

	\begin{proposition}
	\label{df-product-L-algebras}
	Let $A$ and $B$ be two $\mathcal{L}$-algebras.
	Let $P$ be the functor $P:\mathcal{L}^{\text{op}}\to \mathcal{C}$ 
	defined by,

	\begin{enumerate}
	\item $P[n]=A[n]\otimes B[n]$ for all $[n]\in \mathcal{L}$.
	\item $P(\alpha)=A(\alpha)\otimes B(\alpha)$ for all $\alpha$ morphism in $\mathcal{L}$.
	\end{enumerate}

	Let $\mu_\mathcal{P}:\otimes \circ (\mathcal{P}\times \mathcal{P})\to \mathcal{P}\circ +$ 
	be the natural transformation given by the following composition.

		\begin{equation}
		\begin{gathered}
		\xymatrix{
		P[n]\otimes P[m]\ar[r]^-{=} \ar[dd]_-{\mu_\mathcal{P}}&A[n]\otimes B[n]\otimes A[m]\otimes B[m]
		\ar[d]^-{1\otimes T\otimes 1}\\
		&A[n]\otimes A[m]\otimes B[n]\otimes B[m]
		\ar[d]^-{\mu_\mathcal{A}\otimes \mu_\mathcal{B}}\\
		P[n+m]\ar[r]^-{=}&A[n+m]\otimes B[n+m] 
		}
		\end{gathered}
		\end{equation}

	Then $\mathcal{P}$ is an $\mathcal{L}$-algebra.
	\end{proposition}

	\begin{proof}
	Clearly $\mu_{\mathcal{P}}$ satisfy the unit axiom. Commutativity follows from the commutative
	diagram,

	\begin{equation}
	\begin{gathered}
	\xymatrix{
	A[n]\otimes B[n]\otimes A[m]\otimes B[m] \ar[d]_{1\otimes T\otimes 1} \ar[r]^-{T^\sigma}
	& A[m]\otimes B[m]\otimes A[n]\otimes B[n] \ar[d]^-{1\otimes T\otimes 1}\\
	A[n]\otimes A[m]\otimes B[n] \otimes B[m] \ar[r]^-{T\otimes T} \ar@<-1ex>@{<-}[r]_-{T\otimes T}
	\ar[d]_-{\mu_{\mathcal{A}}\otimes \mu_{\mathcal{B}}}
	&  A[m]\otimes A[n]\otimes B[m]\otimes B[n] \ar[d]^-{\mu_{\mathcal{A}}\otimes \mu_{\mathcal{B}}}\\
	A[n+m]\otimes B[n+m] \ar[r]^-{\tau \otimes \tau}& A[n+m]\otimes B[n+m]
	}
	\end{gathered}
	\end{equation}

	where $\sigma=\begin{smatrix}1&2&3&4\\ 3&4&1&2\end{smatrix}$. 
	Upper square is commutative by direct evaluation, 
	and bottom square by commutativity of $\mu_\mathcal{A}$ and $\mu_{\mathcal{B}}$.
	Associativity of $\mu_\mathcal{P}$ can be verified directly,

	\begin{align*}
	\mu_\mathcal{P}(1\otimes \mu_{\mathcal{P}})
	&= (\mu_\mathcal{A}\otimes \mu_\mathcal{B})(1\otimes T\otimes 1)((\mu_\mathcal{A}\otimes \mu_\mathcal{B})(1\otimes T\otimes 1)\otimes 1\otimes 1)\\
	&= (\mu_\mathcal{A}\otimes \mu_\mathcal{B})(1\otimes T\otimes 1)
	(1\otimes 1\otimes (\mu_\mathcal{A}\otimes \mu_{\mathcal{B}})(1\otimes T\otimes 1))\\
	&=\mu_\mathcal{P}(1\otimes \mu_\mathcal{P})
	\end{align*}

	Finally, $\mu_\mathcal{P}$ satisfies the coherence condition 
	because tensor product of two quasi-isomorphisms is again a 
	quasi-isomorphism when $\field$ is a field.
	\end{proof}

	\begin{definition}
	\label{df-tensor-L}
	$\mathcal{L}$-algebra $\mathcal{P}$ defined in \ref{df-product-L-algebras} is
	called the tensor product of $\mathcal{A}$ and $\mathcal{B}$ and
	is denoted $\mathcal{A}\otimes \mathcal{B}$.
	\end{definition}

	\begin{proposition}
	Tensor product of $\mathcal{L}$-algebras extends to
	a functor $\otimes:\mathcal{L}(C)\times \mathcal{L}(C) \to \mathcal{L}(C)$.
	\end{proposition}

	\begin{proof}
	We need to check that for every pair of morphisms of $\mathcal{L}$-algebras,
	$f:\mathcal{A}\to \mathcal{B}$, $g:\mathcal{C}\to \mathcal{D}$, there
	is a morphism of $\mathcal{L}$-algebras $f\otimes g:\mathcal{A}\otimes \mathcal{C}\to \mathcal{B}\otimes \mathcal{D}$.

	Define $f\otimes g$ as $(f\otimes g)_n=f_n\otimes g_n:A[n]\otimes C[n]\to B[n]\otimes D[n]$.
	Let $\alpha:[m]\to [n]$ morphism of $\mathcal{L}$, and consider the following diagram.

	\begin{equation}
	\begin{gathered}
	\xymatrix@C=4pc{
	A[n]\otimes C[n]\ar[r]^-{(f\otimes g)_n} 
	\ar[d]_-{\alpha \otimes \alpha}& B[n]\otimes D[n] \ar[d]^-{\beta \otimes \beta}\\
	A[m]\otimes C[m]\ar[r]^-{(f\otimes g)_n} & B[m]\otimes D[m]
	}
	\end{gathered}
	\end{equation}

	This diagram is commutative because,

	\begin{equation}
	\begin{gathered}
	\begin{array}{lcl}
	(f\otimes g)_m \circ (\alpha \otimes \alpha) &=& (f_m \otimes g_m)\circ(\alpha \otimes \alpha)\\
	&=& (f_m \circ \alpha)\otimes(g_m\otimes \alpha)\\
	&=& (\beta \circ f_n)\otimes (\beta\circ g_n) \text{ ($f$ and $g$ are morphism of $\mathcal{L}$-algebras)}\\ 
	&=& (\beta \otimes \beta)\circ (f\otimes g)_n\\
	\end{array}
	\end{gathered}
	\end{equation}

	$f\otimes g$ preserves the quasi-isomorphism $\mu$ by the following diagram.

	\begin{equation}
	\begin{gathered}
	\xymatrix@C=6pc{
	(\mathcal{A}\otimes \mathcal{C})(n)\otimes (\mathcal{A}\otimes \mathcal{C})(m)
	\ar[r]^-{(f\otimes g)_n\otimes (f\otimes g)_m} \ar[d]^-{\mu_{\mathcal{A}\otimes \mathcal{C}}}&
	(\mathcal{B}\otimes \mathcal{D})(n)\otimes (\mathcal{B}\otimes \mathcal{D})(m)
	\ar[d]^-{\mu_{\mathcal{B}\otimes \mathcal{D}}}\\
	(\mathcal{A}\otimes \mathcal{C})(n+m) \ar[r]^-{(f\otimes g)_{n+m}} &
	(\mathcal{B}\otimes \mathcal{D})(n+m)
	}
	\end{gathered}
	\end{equation}

	Commutativity follows because,

	\begin{equation}
	\begin{gathered}
	\begin{array}{lcl}
	(f\otimes g)_{n+m}\circ \mu_{\mathcal{A}\otimes \mathcal{C}}&=&
	(f_{n+m}\otimes g_{n+m})(\mu_\mathcal{A}\otimes \mu_\mathcal{C})(1\otimes T\otimes 1)\\
	&=& (f_{n+m}\mu_\mathcal{A}\otimes g_{n+m}\mu_{\mathcal{C}})(1\otimes T\otimes 1)\\
	&=& (\mu_\mathcal{B}(f_n\otimes f_m)\otimes \mu_\mathcal{D}(g_n\otimes g_m))(1\otimes T\otimes 1)\\
	&=& (\mu_\mathcal{B}\otimes \mu_\mathcal{D})(f_n\otimes f_m\otimes g_n\otimes g_m)(1\otimes T\otimes 1)\\
	&=& (\mu_\mathcal{B}\otimes \mu_\mathcal{D})(1\otimes T\otimes 1)(f_n\otimes g_n\otimes f_m\otimes g_m)\\
	&=& \mu_{\mathcal{B}\otimes \mathcal{D}}\circ ((f\otimes g)_{n}\otimes (f\otimes g)_m)
	\end{array}
	\end{gathered}
	\end{equation}

	\end{proof}

	Next proposition proof is straightforward.

	\begin{proposition}
	\label{prop-L(C)-strict-monoidal}
	The category $\mathcal{L}(\mathcal{C})$ is a strict symmetric monoidal category with unit,
	where product is given by the tensor product of $\mathcal{L}$ 
	and unit is the trivial $\mathcal{L}$-algebra $\field$.	
	\end{proposition}

	\qed



	\begin{definition}
	\label{df-main-element-l-algebra}
	Let $\mathcal{A}$ be an $\mathcal{L}$-algebra. The module $A[1]$
	is called the main element of $\mathcal{A}$. 
	The associated forgetful functor from $\mathcal{L}(\mathcal{C})$ to $\mathcal{C}$
	is denoted by $U$. In fact we have a collection indexed by $n\geq 0$ of forgetful functors
	$U_n:\mathcal{L}(\mathcal{C})\to \mathcal{C}$, with $U_n(\mathcal{A})=A[n]$.
	\end{definition}

	\begin{definition}
	\label{homology-L-algebras}
	Let $A$ be an $\mathcal{L}$-algebra. The homology of $A$
	is defined to be the homology of its main element.
	\end{definition}

	\begin{remark}
	The homology of $\mathcal{L}$-algebras is equal to the composition of functors $H_*\circ U$,
	where $H_*$ is the homology functor in $\mathcal{C}$. 
	\end{remark}

	\begin{definition}
	\label{quasi-iso-L-algebras}
	Let $f:\mathcal{A}\to \mathcal{B}$ be a morphism of $\mathcal{L}$-algebras
	with values in $\mathcal{C}$.
	The morphism $f$ is called quasi-isomorphism if the induced morphism
	$U(f)$ in $\mathcal{C}$ by the forgetful functor, is a quasi-isomorphism.
	\end{definition}

	\begin{proposition}
	Let $f:\mathcal{A}\to \mathcal{B}$ be a morphism of $\mathcal{L}$-algebras
	with values in $\mathcal{C}$. If $U_k(f)$ is a quasi-isomorphism in $\mathcal{C}$ for some $k$,
	then $U_{kn}(f)$ is a quasi-isomorphism for every $n\geq 0$.
	In particular if $f$ is a quasi-isomorphism then $U_n(f)$
	is a quasi-isomorphism for every $n\geq 1$.
	\end{proposition}

	\begin{proof}
	We proceed by induction. The hypothesis says that $f_k:A[k]\to B[k]$
	is a quasi-isomorphism. Now, the following diagram is commutative
	because $f$ is a morphism of $\mathcal{L}$-algebras

	\begin{equation}
	\begin{gathered}
	\xymatrix@C=6pc{
	A[k]\otimes A[k(n-1)] \ar[r]^-{f_k\otimes f_{k(n-1)}} \ar[d]_-{\mu_\mathcal{A}} 
	& B[k]\otimes B[k(n-1)] \ar[d]^-{\mu_\mathcal{B}}\\
	A[kn]\ar[r]^-{f_{kn}} & B[kn]
	}
	\end{gathered}
	\end{equation}

	The tensor product $f_k\otimes f_{k(n-1)}$ is a quasi-isomorphism since
	$\field$ is a field. Then $f_{kn}$ is a quasi-isomorphism.

	\end{proof}

	\begin{definition}
	\label{df-quasi-isomor-relation}
	The equivalence relation on $\mathcal{L}$-algebras spanned by quasi-isomorphisms
	will be called again quasi-isomorphism.
	\end{definition}


	The concept of $\mathcal{L}$ is inspired by the fact that 
	any natural diagonal of chain complexes $C_*(X)\to C_*(X)\otimes C_*(X)$, 
	is determined by a zig-zag of natural morphisms 
	$C_*(X)\to C_*(X\times X) \leftarrow C_*(X)\otimes C_*(X)$,
	where the first arrow is the morphism induced 
	by the simplicial diagonal $X\to X\times X$
	and the second arrow is the Eilenberg-Mac~Lane transformation. 
	In this section we will proceed to describe the $\mathcal{L}$-algebra structure 
	on the chain complexes that have as product the Eilenberg-Mac~Lane transformation,
	using a canonical way to associate 
	to each simplicial set (not necessarily $0$-reduced)
	an $\mathcal{L}$-algebra whose main element is its chain complex.

	\begin{proposition}
	\index{pointedsimplicialset@Pointed simplicial set}
	Let $\sset_*$ be the category of pointed simplicial set. Then for every
	simplicial set $X$, there is an unique
	functor $S_X:\mathcal{L}^{\text{op}}\to \sset_*$ preserving zeros,
	mapping sums to products, and $[1]$ to $X$.
	\end{proposition}

	\begin{proof}
	$\sset_*$ is a strict pointed cartesian category, then by proposition \ref{free-cocartesian-category}
	the result follows.
	\end{proof}

	\begin{proposition}
	\label{prop-L-alg-canonique}
	Let $X$ be a pointed simplicial set. Then composition,

	\begin{equation}
	\begin{gathered}
	C_*\circ S_X:\mathcal{L}^{\text{op}}\to \mathcal{C}
	\end{gathered}
	\end{equation}

	together with the Eilenberg-Mac~Lane transformation is an $\mathcal{L}$-algebra.
	\end{proposition}

	\begin{proof}
	It follows easily from Eilenberg-Mac~Lane transformation properties.
	\end{proof}

	\begin{definition}
	\label{df-canonical-L-alg}
	The $\mathcal{L}$-algebra associated to the pointed simplicial set $X$
	will be called the canonical $\mathcal{L}$-algebra of $X$, and
	denoted by $\mathcal{A}_X$.
	\end{definition}

	$\mathcal{L}$-algebra $\mathcal{A}_X$ for $n\geq 1$ 
	take the values $\mathcal{A}_X[n]=C_*(X^n)$
	with product
	$\mu_{\mathcal{A}_X}$
	the Eilenberg-Mac~Lane transformation
	$\nabla_{n,m}:C_*(X^n)\otimes C_*(X^m)\to C_*(X^{n+m})$
	Let $*$ be the base point of $X$, then the image by $\mathcal{A}_X$ for
	a morphism $\alpha:[m]\to [n]$ in $\mathcal{L}$ is given by the following formula.

	\begin{equation}
	\begin{gathered}
	\xymatrix@C=6pc@R=0.2pc{
	\mathcal{A}_X[n]=C_*(X^n)\ar[r]^-{\mathcal{A}_X(\alpha)} & \mathcal{A}_X[m]=C_*(X^m)\\
	(x_1,\ldots,x_n) \ar@{{|}-{>}}[r] & (x_{\alpha(1)},\ldots,x_{\alpha(m)})\\
	}
	\end{gathered}
	\end{equation}

	Where $x_{\alpha(j)}=*$ for each $j$ not in $Dom(\alpha)$.
	Note that taking
	$X=*$ the simplicial point, $\mathcal{A}_*$ is the trivial $\mathcal{L}$-algebra $\field$.

	In the case of a simplicial group (who will be pointed by its unit), we have an extra
	structure.

	\begin{proposition}
	Let $H$ be a simplicial group, then $\mathcal{A}_H$
	is a multiplicative $\mathcal{L}$-algebra.
	\end{proposition}

	\begin{proof}
	For every $n\geq 0$, $H^n$ and
	$\mathcal{A}_H[n]=C_*(H^n)$ is a differential graded algebra (with
	the Pontrjagin product). Since the Eilenberg-Mac~Lane is a morphism of algebras, the functor
	$\mathcal{A}$ maps simplicial groups to multiplicative $\mathcal{L}$-algebras.
	\end{proof}




\section{Principal results}
\label{section-principal-results}

\subsection{The Operad $\mathcal{K}$}
\label{subsection-construction-operad-K}

	In this section is constructed an $E_\infty$-operad $\mathcal{K}$
	by using the inductive colimit
	of a suitable collection of operads $\{\mathcal{K}_n\}_{n\geq 2}$.

	This construction starts with a $\mathbb{S}$-module $M$ concentrated in arity $2$
	and such that $M(2)$ is suppose to be a 
	$\dga$-module $\field[\Sigma_2]$-free resolution of $\field$.
	Operad $\mathcal{K}_2$ is taken as the free operad on $M$,
 	which makes that $\mathcal{K}_2$ will have like $M$ 
	a $\field[\Sigma_2]$-free resolution of $\field$ 
	with all abstract binary operations coded by $\mathcal{K}_2$. 	

	However, in higher arities $\mathcal{K}_2$  fails to have free resolutions of $\field$. 	
	For example, there is no guarantee we can found in $\mathcal{K}_2(3)$
	appropriate homotopies linking trees of degree $0$
	with nodes any generating element $x$ of degree $0$ in $\mathcal{K}_2(2)$, 
	like the ones in the following picture.

		\begin{equation}
		\begin{gathered}
		\begin{tikzpicture}
		\foreach \h in {1}{
			\foreach \l in {1}{
			\foreach \x in {-2}{
			\foreach \y in {0}{
			\draw [blue, thick, fill] (\x,\y+\h/2) circle [radius=0.05];
			\draw [-, thick](\x-\l/2,\y+\h) to  [out=270, in=100] (\x,\y+\h/2);
			\node at (\x-\l/2,\y+\h+0.3) {$1$};
			\draw [-, thick](\x+\l/2,\y+\h) to  [out=270, in=80] (\x,\y+\h/2);
			\draw [-, thick](\x,\y) to  (\x,\y+\h/2);
			\node at (\x+0.4,\y+\h/2) {$x$} ;
			}
			}
			}
		}
		\foreach \h in {1}{
			\foreach \l in {0.8}{
			\foreach \x in {-1.5}{
			\foreach \y in {0.5}{
			\draw [blue, thick, fill] (\x,\y+\h/2) circle [radius=0.05];
			\draw [-, thick](\x-\l/2,\y+\h) to  [out=270, in=100] (\x,\y+\h/2);
			\node at (\x-\l/2,\y+\h+0.3) {$2$};
			\draw [-, thick](\x+\l/2,\y+\h) to  [out=270, in=80] (\x,\y+\h/2);
			\node at (\x+\l/2,\y+\h+0.3) {$3$};
			\node at (\x+0.4,\y+\h/2) {$x$} ;
			}
			}
			}
		}
		\node at (-0.25,0.5) {$\longleftrightarrow$} ;
		\foreach \h in {1}{
			\foreach \l in {1}{
			\foreach \x in {1.5}{
			\foreach \y in {0}{
			\draw [blue, thick, fill] (\x,\y+\h/2) circle [radius=0.05];
			\draw [-, thick](\x-\l/2,\y+\h) to  [out=270, in=100] (\x,\y+\h/2);
			\node at (\x+\l/2,\y+\h+0.3) {$3$};
			\draw [-, thick](\x+\l/2,\y+\h) to  [out=270, in=80] (\x,\y+\h/2);
			\draw [-, thick](\x,\y) to  (\x,\y+\h/2);
			\node at (\x+0.4,\y+\h/2) {$x$} ;
			}
			}
			}
		}
		\foreach \h in {1}{
			\foreach \l in {0.8}{
			\foreach \x in {1}{
			\foreach \y in {0.5}{
			\draw [blue, thick, fill] (\x,\y+\h/2) circle [radius=0.05];
			\draw [-, thick](\x-\l/2,\y+\h) to  [out=270, in=100] (\x,\y+\h/2);
			\node at (\x-\l/2,\y+\h+0.3) {$1$};
			\draw [-, thick](\x+\l/2,\y+\h) to  [out=270, in=80] (\x,\y+\h/2);
			\node at (\x+\l/2,\y+\h+0.3) {$2$};
			\node at (\x-0.4,\y+\h/2) {$x$} ;
			}
			}
			}
		}
		\end{tikzpicture}
		\end{gathered}
		\end{equation}

	Next step extend operad $\mathcal{K}_2$		
	into an operad $\mathcal{K}_3$ having on component $\mathcal{K}_2(3)$
	all missing homotopies by using  
	the polynomial operads technique 
	discussed in subsection \ref{subsection-polynomial-operads}.
	
	This process is repeated with each component
	in order to include all possible homotopies,
	creating a nested collection of operad
	which inductive limit will have all required homotopies
	to form an $E_\infty$-operad, the operad $\mathcal{K}$.

	\begin{definition}
	\label{df-acyclic-extension}
		Let $M$ be a $\field[\Sigma_n]$-free finitely generated $\dga$-module.
		The acyclic extension of $M$ is the associated acyclic $\dga$-module 
		given by proposition \ref{prop-acyclic-complete}. 
		It will be denoted by $X(M)$.
	\end{definition}
	
	\begin{definition}
	\label{df-collection-Kn}
		Let $M$ be a $\dga$-module $\field[\Sigma_2]$-free resolution of $\field$.
		For $n\geq 2$, $\mathcal{K}_n$ is the operad defined by induction as follows.

		\begin{enumerate}
		\item $\mathcal{K}_2=F(M)$, 
		the free operad on $M$, with $M$ seen as $\mathbb{S}$-module concentrated in arity $2$.
		\item $\mathcal{K}_{n+1}=\mathcal{K}_n[M_n]$, 
		the polynomial operad on $M_n$ with coefficients in $\mathcal{K}_n$,
		defined in proposition \ref{th-adj-poly-operads},
		and $M_n$ is the $\mathbb{S}$-module given by:
			\begin{equation}
			\begin{gathered}
				M_n(i) =
				  \begin{cases}
				    \mathcal{K}_n(i)       & \quad i\neq n+1\\
				    X(\mathcal{K}_n(n+1))  & \quad i=n+1\\
				  \end{cases}
			\end{gathered}
			\end{equation}
		\end{enumerate}
	\end{definition}

		$K_2=F(M)$ is only formed from compositions 
		of arity $2$ operations in $M$,
		and is not acyclic for arities $\geq 3$. 
		$K_2(3)$ is extended to an acyclic $\dga$-module $X(K_2(3))$,
		Inclusion $K_2 \hookrightarrow M_2$,
		which is strict only in arity $3$, 
		increase $K_2$ with new operations in arity $3$ 
		not decomposable in terms of operations of arity $2$.
		Now, operad $K_3=K_2[M_2]$ is formed
		by compositions all operations in $K_2$ and the new arity $3$ operations,
		with $K_3$ matching $K_2$ in arity $2$, and in general, 
		the extension $K_n\to K_{n+1}$ is the identity for arities $\leq n$.

		Formally, canonical inclusions
		$\mathcal{K}_n\hookrightarrow \mathcal{K}_{n+1}$
		are given by arrows $\alpha_n$
		from the construction of each $\mathcal{K}_{n+1}$.

			\begin{equation}
			\begin{gathered}
				\xymatrix{
				FU(\mathcal{K}_n) \ar[r]^-{\epsilon_{\mathcal{K}_n}} \ar[d]_-{F(i_{\mathcal{K}_n})} 
				& \mathcal{K}_n \\
				F(M_n) 
				& \mathcal{K}_n[M_n]=\mathcal{K}_{n+1} \ar@{<-}[u]_-{\alpha_n} \ar@{<-}[l]^-{\beta_n} \\
				}
			\end{gathered}
			\end{equation}

		But
		$\alpha_n:\mathcal{K}_n\to \mathcal{K}_{n+1}$
		behaves like the identity for arities $\leq n$
		in the sense of lemma \ref{lemma-alpha-iso},
		allowing the following definition.
	

	\begin{definition}
	\label{df-operad-K}
		The operad $\mathcal{K}$ is defined to be 
		the inductive limit of the collection of operads,
			\begin{equation}
			\begin{gathered}
				\xymatrix{
				\mathcal{K}_2 \ar@{^{(}->}[r]^-{\alpha_2} 
				&\mathcal{K}_3 \ar@{^{(}->}[r]^-{\alpha_3} &\cdots
				\ar@{^{(}->}[r]^-{\alpha_{n-1}} 
				&\mathcal{K}_n \ar@{^{(}->}[r]^-{\alpha_n} & \cdots
				}
			\end{gathered}
			\end{equation}
	\end{definition}

	\begin{proposition}
		The operad $\mathcal{K}$ is an $E_\infty$-operad.
	\end{proposition}

	\begin{proof}
		We have $\mathcal{K}(2)=\mathcal{K}_2(2)=M_2$, 
		a $\field[\Sigma_2]$-free resolution of $\field$, 
		and by construction 
		$\mathcal{K}(n+1)=\mathcal{K}_{n+1}(n+1)=M_n=X(\mathcal{K}_n(n+1))$,
		which is acyclic and $\field[\Sigma_{n+1}]$-free.
	\end{proof}

\subsection{$E_\infty$-Structures in $\mathcal{L}$-Algebras}

	Following theorem is the principal objective of this paper,
	it exhibits the nature of main elements of $\mathcal{L}$-algebras
	as $E_\infty$-coalgebras. 
	The outline of the proof is, 
	we start by showing that main element $A[1]$ of an $\mathcal{L}$-algebra $\mathcal{A}$ 
	is a $\mathcal{K}_n$-coalgebra for all $n>1$. 
	Then, using the fact that operad $\mathcal{K}$ 
	is the inductive limit of these operads, 
	$A[1]$ can have an induced $\mathcal{K}$-coalgebra structure, 
	in other words, $A[1]$ is an $E_\infty$-coalgebra.

	\begin{theorem}[Main Theorem]
	\label{main-theorem}
		Let $\mathcal{K}$ be the $E_\infty$-operad defined in \ref{df-operad-K}.
		Then, there exists a functor  $\mathcal{F}$ 
			\begin{equation}
			\begin{gathered}
				\xymatrix{
				\mathcal{L}\text{-Alg} \ar[r]^-{\mathcal{F}} 
				& \mathcal{K}\text{-CoAlg}
				}
			\end{gathered}
			\end{equation}
		which associates a $\mathcal{K}$-coalgebra 
		$\mathcal{F}(\mathcal{A})$ to each $\mathcal{L}$-algebra $\mathcal{A}$ 
		and satisfies the following conditions.
			\begin{enumerate}
				\item The underlying $\dga$ module 
				of $\mathcal{F}(\mathcal{A})$ is $A[1]$.
				\item For every $n\geq 1$, 
				the morphism $\varphi_n:K(n)\otimes A[1]\to A[1]^{\otimes n}$,
				given by the $\mathcal{K}$-coalgebra structure 
				defined on $A[1]$ by $\mathcal{F}$, 
				makes the following diagram commutative up to homotopy,
					\begin{equation}\label{dgm-condition-2-main-th}
					\begin{gathered}
						\xymatrix{
						A[1]^{\otimes n}
						\ar[rr]^-{\mu} 
						& & A[n]\\
						& K(n)\otimes A[1] 
						\ar[ur]_-{s_1(\epsilon \otimes 1)}
						\ar[ul]^-{\varphi_n}
						&
						}
					\end{gathered}
					\end{equation}
				where $\mu$ is given by 
				the structural quasi-isomorphism of $\mathcal{A}$ 
				and $s_1$ is the image
				by $\mathcal{A}$ of the only morphism in $\mathcal{L}$ 
				of the form $([n],\alpha):[n]\to [1]$.
			\end{enumerate}
	\end{theorem}

	\begin{proof}[Proof of theorem \ref{main-theorem}]
		We use the fact that the operad $E_\infty$-operad $\mathcal{K}$ 
		is the inductive limit of the sequence of operads,
			\begin{equation}
			\begin{gathered}
				\mathcal{K}_2 \subset \cdots \subset \mathcal{K}_n \subset \cdots \subset \mathcal{K}
			\end{gathered}
			\end{equation}
		in order to proceed by induction. 
		We first show for all $n\geq 2$ 
		that $A[1]$ has an structure of $\mathcal{K}_n$-coalgebra 
		which satisfies the second condition of the theorem.
		That is, there exists an operad morphism 
		$\overline{F}_n:\mathcal{K}_n\to \text{Coend}(A[1])$,
		such that the associated morphisms
		 $\varphi_i:K_n(i)\otimes A[1]\to A[1]^{\otimes i}$,
		makes the following diagram commutative up to homotopy,
			\begin{equation}\label{dgm-condition-2-main-th}
			\begin{gathered}
				\xymatrix{
				A[1]^{\otimes i}\ar[rr]^-{\mu} & & A[i]\\
				& K_n(i)\otimes A[1] \ar[ur]_-{s_1(\epsilon \otimes 1)}\ar[ul]^-{\varphi_i}&
				}
			\end{gathered}
			\end{equation}
		where $\mu$ is given by the structural quasi-isomorphism of $\mathcal{A}$ 
		and $s_1$ is the image by $\mathcal{A}$ of 
		the only morphism in $\mathcal{L}$ of the form $([i],\alpha):[i]\to [1]$.

		\textbf{Case $\mathcal{K}_2$: } 
		Recall that $\mathcal{K}_2$ is the free operad on the 
		$\mathbb{S}$-module $M_2$ concentrated in arity $2$.
		To show that $A[1]$ is a $\mathcal{K}_2$-coalgebra, 
		we define an $\Sigma_2$-equivariant morphism from $M_2(2)$ to $\text{Coend}(A[1])(2)$ 
		using the relative lifting theorem \ref{th-relative-relevement} in order to 
		satisfy the condition
		on $\mathcal{K}_2$ and then, the $\mathcal{K}_2$-coalgebra 
		structure is obtained as a consequence of 
		the universal property of free operads.

		Defining a $\Sigma_2$ morphism from $M_2(2)$ to $\text{Coend}(A[1])(2)$ 
		is equivalent to define a morphism of $\dga$-$\field[\Sigma_2]$-modules,
			\begin{equation}
			\begin{gathered}
				\overline{\varphi}_2:K_2(2)\otimes A[1]\to A[1]\otimes A[1]
			\end{gathered}
			\end{equation}
		Recall that $M_2(2)=K_2(2)$, 
		the action of $\Sigma_2$ on $A[1]\otimes A[1]$ 
		is the permutation of factors 
		and the action of $\Sigma_2$ on $K_2(2)\otimes A[1]$ 
		maps $x\otimes a$ to $x \sigma \otimes a$.
		Now consider the following diagram,
			\begin{equation}
			\begin{gathered}
				\xymatrix{
				A[1]\otimes A[1]\ar[r]^-{\mu} & A[2] \\
				& K_2(2)\otimes A[1] 
				\ar[u]_-{s_0\circ (\epsilon \otimes 1)} 
				\ar@{..{>}}[ul]^-{\varphi_2}
				}
			\end{gathered}
			\end{equation}
		where $\epsilon$ is the augmentation of $K_2(2)$, 
		$s_0:A[1]\to A[2]$ is the image by $\mathcal{A}$ 
		of the only arrow in $\mathcal{L}$ 
		of the form $([2],\alpha):[2]\to [1]$ 
		and $\mu$ is the structural 
		quasi-isomorphism of $\mathcal{A}$. 

		The $\dga$-$\field[\Sigma_2]$-morphism $\varphi_2$ 
		that makes the diagram commutative up to homotopy 
		is obtained with the theorem \ref{th-relative-relevement} 
		by taking $L'=0$.
		This complete the existence of a $\Sigma_2$-equivariant morphism 
		from $M_2(2)$ to $\text{Coend}(A[1])(2)$ and therefore, 
		we have a morphism $F_2$ of $\mathbb{S}$-modules 
		from $M_2$  to $\text{Coend}(A[1])$,
		which behaves on $M_2(2)$ as $\varphi_2$ 
		and as $0$ on $M_2(i)$, $i\neq2$. 

		Now, consider the following diagram,

		\begin{equation}
		\begin{gathered}
			\xymatrix{
			M_2 \ar@{^{(}->}[r] \ar[rd]_-{F_2} 
			& \mathcal{K}_2=F(M_2) \ar@{-->}[d]^-{\overline{F}_2}\\
			& \text{Coend}(A[1])
			}
		\end{gathered}
		\end{equation}

		where the upper arrow is given by the inclusion of $\mathbb{S}$-modules. 
		The universal property of the free operad $\mathcal{K}_2$ says
		there is an unique morphism of operads $\overline{F}_2$ making 
		the diagram commutative. 
		This morphism $\overline{F_2}$ 
		gives the $\mathcal{K}_2$-coalgebra structure 
		on $A[1]$ that we wanted. 

		\textbf{Case $\mathcal{K}_n$: } 
		Suppose we have a sequence of operad morphisms 
		$\overline{F}_2,\ldots,\overline{F}_{n-1}$, 
		such that, 
		for $i<n$, $\overline{F}_i:\mathcal{K}_i\to \text{Coend}(A[1])$ 
		and $\overline{F}_i$ satisfies the second condition of the theorem.
		As we have seen in the construction of $K_i$'s, 
		the operad $\mathcal{K}_{n-1}$ as $\mathbb{S}$-module,
		can be embedded as a direct factor in a $\mathbb{S}$-module $M_{n-1}$ 
		with component $M_{n-1}(n)$ acyclic and $\field[\Sigma_n]$-free.
		Then we have $M_{n-1}=\mathcal{K}_{n-1}\oplus P$, 
		where the component $P(j)$ is $\field[\Sigma_j]$-free for $j>0$ .

		Observe that this defines an 
		object $(\mathcal{K}_{n-1},M_{n-1})$ 
		and morphism in $\mathfrak{C}$,

		\begin{equation}
		\begin{gathered}
			F_{n-1}:(K_{n-1},M_{n-1})\to(\text{Coend}(A[1]),U(\text{Coend}(A[1])))
		\end{gathered}
		\end{equation}

		which behaves as $\overline{F}_{n-1}$ 
		on $\mathcal{K}_{n-1}$ and as $0$ on $P$. 
		In order to satisfy the
		second condition of the theorem we focus our attention 
		in the $\Sigma_j$-equivariant 
		morphism given by $\overline{F}_{n-1}$ 
		on the component $\mathcal{K}_{n-1}(j)$, $j>0$. 
		We will extend this morphism on the components $M_{n-1}(j)$ 
		or equivalently, 
		define a $\dga$-$\field[\Sigma_j]$ morphism $\overline{\phi}_j$ 
		from $M_{n-1}(j)\otimes A[1]$ to $A[1]^{\otimes j}$. 
		In order to do that, consider the diagram,

		\begin{equation}
		\begin{gathered}
			\xymatrix{
			A[1]^{\otimes j}\ar[rr]^-\mu & & A[j] \\
			& M_{n-1}(j)\otimes A[1]  
			\ar@{..>}[ul]_-{\overline{\phi}_j}  
			\ar[ur]^-{s_0\circ (\epsilon \otimes 1)}& \\
			& K_{n-1}(j)\otimes A[1] 
			\ar@/^1pc/[uul]^-{\phi_j} 
			\ar@{^{(}->}[u] 
			\ar@/_1pc/[uur]_-{s_0\circ (\epsilon \otimes 1)}& 
			}
		\end{gathered}
		\end{equation}

		where $\phi_j$ is the $\field[\Sigma_n]$-morphism 
		induced by $\overline{F}_{n-1}$.
		By hypothesis, the morphism $\phi_j$ makes commutative 
		up to homotopy the outer triangle of the diagram.
		Then the existence of $\overline{\phi}_j$ 
		follows after applying the relative lifting theorem \ref{th-relative-relevement}
		with $L=M_{n-1}(j)\otimes A[1]$ and $L'=K_{n-1}(j)\otimes A[1]$.

		Observe that the $\Sigma_j$-equivariant 
		morphism from $M_{n-1}(j)$ to $\text{Coend}(A[1])(j)$ 
		induced by $\overline{\phi}_j$ 
		behaves like $\overline{F}_{n-1}$ over $K_{n-1}(j)$. 
		Denote by $F_{n-1}$ the morphism of 
		$\mathbb{S}$-modules given by this data.
		Then $(\overline{F}_{n-1},F_{n-1}):
		(\mathcal{K}_{n-1},M_{n-1})\to (\text{Coend}(A[1]),$ 
		$U(\text{Coend}(A[1])))$ 
		is a morphism of $\mathfrak{C}$, 
		and consider the following diagram.

			\begin{equation}
			\begin{gathered}
				\xymatrix{
				(\mathcal{K}_{n-1},M_{n-1}) 
				\ar[rd]_-{F_{n-1}} 
				\ar[r]^-{\Psi}& 
				(\mathcal{K}_{n-1}[M_{n-1}],U(\mathcal{K}_{n-1}[M_{n-1}])) 
				\ar@{-->}[d]^-{(\overline{F}_n,U(\overline{F}_n)}\\
				& (\text{Coend}(A[1]),U(\text{Coend}(A[1])))
				}
			\end{gathered}
			\end{equation}

		The operad morphism $\overline{F}_{n}$ 
		making the diagram commutative follows 
		by proposition \ref{prop-property-unit-adjunction-poly-op}.
		Then, $\overline{F}_n$ gives 
		the $\mathcal{K}_n$-coalgebra structure on $A[1]$ needed
		to complete the inductive step.

		\textbf{$A[1]$ is a $\mathcal{K}$-coalgebra: }
		We now proceed with the final part of the proof 
		and exhibit $A[1]$ as a $\mathcal{K}$-coalgebra. 
		Consider the following cocone of operads
		onthe diagram given by the operads 
		$\{\mathcal{K}_i\}_{i\geq 2}$.

			\begin{equation}
			\begin{gathered}
				\xymatrix{
				\mathcal{K} 
				\ar@{-->}[rrrr]^-{\overline{F}} 
				&&&& \text{Coend}(A[1])\\
				&&&\\
				*++{\mathcal{K}_2} 
				\ar[uurrrr]_-{\overline{F}_2} 
				\ar[uu] \ar@<-0.5ex>@{^{(}->}[r]  & 
				\cdots 
				\ar@{}[uurrr]_-{\cdots} 
				\ar@<-0.5ex>@{^{(}->}[r] 
				\ar@{}[uul]^-{\cdots}& 
				*++{\mathcal{K}_n} 
				\ar@<-0.5ex>@{^{(}->}[r] 
				\ar[uurr]_-{\overline{F}_n} 
				\ar[uull]& \cdots
				\ar@{}[uur]_-{\cdots}
				}
			\end{gathered}
			\end{equation}

		The inductive part of the proof 
		exhibited the operad $\text{Coend}(A[1])$, 
		together with the morphisms $\overline{F}_n$'s,
		as a cocone on the $\mathcal{K}_i$'s. 
		By definition $\mathcal{K}$ 
		is also a cocone on the $\mathcal{K}_i$'s. 
		Then the universal property of colimits says 
		that there exists an unique morphism 
		of operads $\overline{F}$
		from $\mathcal{K}$ to $\text{Coend}(A[1])$ 
		commutative on these two cocones.
		The morphism $\overline{F}:\mathcal{K}\to \text{Coend}(A[1])$ 
		exhibit $A[1]$ as an $\mathcal{K}$-coalgebra
		with the conditions stated by the theorem.

		\textbf{Functoriality:} 
		Let $f:\mathcal{A}\to \mathcal{B}$ 
		be a morphism of $\mathcal{L}$ and consider
		the following diagram.

			\begin{equation}
			\begin{gathered}
				\xymatrix{
				A[n]^{\otimes n} 
				\ar[rd]_-{\mu^{\mathcal{A}}}
				\ar[rrr]^-{f^{\otimes n}_1}
				&&& 
				B[n]^{\otimes n} 
				\ar[dl]^-{\mu^\mathcal{B}}\\
				&A[n]\ar[r]^-{f_n}&B[n]&\\
				K[n]\otimes A[1] 
				\ar[rrr]_-{1\otimes f_1}
				\ar[ru]^-{s(\epsilon\otimes 1)} 
				\ar@/^1pc/[uu]^-{\varphi_n^{\mathcal{A}}}
				&&& 
				K[n]\otimes B[1] 
				\ar[lu]_-{s(\epsilon\otimes 1)} 
				\ar@/_1pc/[uu]^-{\varphi_n^{\mathcal{B}}}
				}
			\end{gathered}
			\end{equation}

		The two triangles are commutative up to homotopy 
		by the second condition of the theorem and 
		the inner diagrams are commutative because 
		$f$ is a morphism of $\mathcal{L}$-algebras. 
		The commutative up to homotopy of the outer diagram 
		follows from this and the fact that 
		$\mu$ is a quasi-isomorphism. 
		This shows that our construction is 
		functorial and completes the proof.
		\end{proof}

\section{Funding}

This work was supported by Universidad de Costa Rica [grant number OAICE-08-CAB-144-2010]
and by the Vicerrector\'ia de Investigaci\'on de la Universidad de Costa Rica through the project 821-B6-A19.
I thank Alain Prout\'e for encouraging me to think about this problem.

\bibliographystyle{amsplain}
\bibliography{principal-bibliography}

\end{document}